\documentclass{aims}
\usepackage{amsmath,amsfonts,amssymb,amsthm}
  \usepackage{paralist}
  \usepackage{graphicx,xcolor}
  \usepackage{graphics} 
  \usepackage{epsfig} 
\usepackage{epstopdf}
 \usepackage[colorlinks=true]{hyperref}
\hypersetup{urlcolor=blue, citecolor=red}

\def\currentvolume{X}
 \def\currentissue{X}
  \def\currentyear{200X}
   \def\currentmonth{XX}
    \def\ppages{X--XX}
     \def\DOI{10.3934/xx.xx.xx.xx}
     
\newcommand{\R}{\ensuremath{\mathbb{R}}}

\newcommand{\ep}{\varepsilon}

\newcommand{\de}{\delta}
\newcommand{\De}{\Delta}
\newtheorem{theorem}{Theorem}[section]

\newtheorem{proposition}{Proposition}

\theoremstyle{definition}

\newtheorem{remark}{Remark}

\title[Shadowing near nonhyperbolic fixed points]
      {Shadowing near nonhyperbolic fixed points}

\author[Alexey A.\ Petrov and Sergei Yu.\ Pilyugin]{}

\subjclass{Primary: 37C50.} 
 \keywords{Dynamical system, shadowing, fixed point,
 Lyapunov function.}

 \email{al.petrov239@gmail.com}
 \email{sp@sp1196.spb.edu}

\thanks{Supported by the Russian Foundation for Basic Research
(project 12-01-00275); the first author is also supported by the Chebyshev
Laboratory, Faculty of Mathematics and Mechanics, St.\ Petersburg State 
University, under grant 11.G34.31.0026 of the Government of Russian Federation.}

\begin{document}
\maketitle

\centerline{\scshape Alexey A.\ Petrov }
\medskip
{\footnotesize
 \centerline{Faculty of Mathematics and Mechanics St.\ Petersburg State 
 University }
   \centerline{University av.,\ 28, 198504, St.\ Petersburg, Russia}
} 

\medskip

\centerline{\scshape Sergei Yu.\ Pilyugin}
\medskip
{\footnotesize
 \centerline{ Faculty of Mathematics and Mechanics St.\ Petersburg State 
  University}
   \centerline{University av.,\ 28, 198504, St.\ Petersburg, Russia,}
   \centerline{St.\ Petersburg State Electrotechnical University (LETI)}
   \centerline{Popova str., 5, 197376, St.\ Petersburg, Russia}
}

\bigskip

 \centerline{(Communicated by Lan Wen)}

\begin{abstract}
We use Lyapunov type functions to find conditions
of finite shadowing in a neighborhood of a nonhyperbolic
fixed point of a one-dimensional or two-dimensional
homeomorphism or diffeomorphism. A new concept of shadowing
in which we control the size of one-step errors 
 is introduced in the case of a nonisolated
fixed point.
\end{abstract}

\section{Introduction}

The shadowing property of dynamical systems (diffeomorphisms or flows)
is now well-studied (see, for example, the monographs [1, 2] and the
recent survey [3]).
This property means that, near approximate trajectories
(so-called pseudotrajectories), there exist exact trajectories of the
system.

In this paper, we are interested in shadowing near fixed points.

Mostly, standard methods allow one to show that 
a diffeomorphism has the shadowing property near a
hyperbolic fixed point, and this property is Lipschitz, see [1].

One can mention several papers which contain methods of proving
the shadowing property for systems with nonhyperbolic behavior
(see, for example, [4, 5]).

Our approach is based on the method of Lyapunov type functions.
First, Lyapunov type functions were used in the study of
shadowing and topological stability by Lewowicz in [6].
We apply results on Lyapunov type functions obtained in our
paper [7].

Let us state the problem of shadowing in general.

Let $f$ be a homeomorphism of a metric space $(M,\mbox{dist})$.

In this paper, we define a finite $d$-pseudotrajectory of $f$
as a set of points $\{p_k\in M:\;0\leq k\leq m\}$ such that
$$
\mbox{dist}(f(p_k),p_{k+1})\leq d, \quad 0\leq k\leq m-1.
$$

In the study of shadowing in noninvariant sets (such as
neighborhoods of fixed points), the concept of finite
shadowing is natural.

We say that $f$ has the finite shadowing property in a set $K\subset M$ if,
for any $\ep>0$, there exists a $d>0$ such that if $\{p_k\in K:
\;0\leq k\leq m\}$ is a $d$-pseudotrajectory of $f$, then there
exists a point $r$ such that
\begin{equation}
\label{0}
\mbox{dist}(f^k(r),p_k)\leq\ep, \quad 0\leq k\leq m.
\end{equation}

Let us emphasize that in the above definition, $d$ depends on $K$
and $\ep$ but not on $m$.

The structure of the paper is as follows.
In Sec. \ref{sec2}, we treat the one-dimensional case. We prove a simple general
statement (Theorem \ref{res1}) and show that, in some cases, the dependence of
$d$ on $\ep$ can be clarified. Section \ref{sec3} is devoted to the method of
Lyapunov type functions developed in [7]. In Sec. \ref{sec4}, we give
general conditions of finite shadowing in the two-dimensional case
and then treat in detail an important example of a diffeomorphism 
of the form
\begin{equation}
\label{3.0}
f(x,y)=(x-x^{2n+1}+X(x,y),y+y^{2m+1}+Y(x,y)),
\end{equation}
where $n,m$ are natural numbers
and $X,Y$ are smooth functions that vanish at the origin
together with their Jacobi matrices.

Finally, Sec. \ref{sec5} is devoted to shadowing near a nonisolated fixed
point. We study a simple (but nontrivial) example of
the diffeomorphism
\begin{equation}
\label{4.0}
f(x,y)=\left(\frac{x}{2}, y(1+x^2)\right),
\end{equation}
for which the origin is a nonisolated fixed point (any point
$(0,y)$ is a fixed one).

Of course, such a system does not have the usual shadowing
property. In this case, we work with a new concept of shadowing
in which we control the size of one-step errors.

Our methods can be applied to dynamical systems with
phase space of arbitrary dimension; in this paper, we restrict
the consideration to one-dimensional and two-dimensional
systems to clarify the presentation of the main ideas.

\section{One-dimensional case}\label{sec2}

First we consider the problem of shadowing near a nonhyperbolic
fixed point of a homeomorphism in the simplest, one-dimensional,
case.

Let $f$ be a homeomorphism of a neighborhood $U$ of a fixed
point $0\in\R$ to its image. 

We consider in detail the case where $f$ is nonhyperbolically 
expanding in a neighborhood of a fixed point; the case of
nonhyperbolic contraction near a fixed point is treated
similarly.

We impose simplest possible conditions on $f$; 
in a sense, precisely this topological form of conditions 
is generalized by our conditions
in the two-dimensional case.
\medskip

{\bf Condition 1. } There exist numbers $a,A>0$ such that
if $|x|\leq A$ and $0<v<a$, then
\begin{equation}
\label{1}
f(x+v)-f(x)>v,\quad f(x-v)-f(x)<-v.
\end{equation}

Denote by $B(r,A)$ the closed $r$-neighborhood of a set $A$.
\medskip
\begin{theorem}\label{res1}
{ If condition} {\em 1} { is satisfied, then $f$ has
the finite shadowing property in the set} ${\mathcal{B}}=B(A,0)$
\end{theorem}
\begin{proof}
Fix $\ep>0$; we assume that $\ep\leq a$.

Condition 1 implies that if $x\in{\mathcal{B}}$, then
$$
B(\ep,f(x))\subset \mbox{Int} f(B(\ep,x)).
$$

The compactness of the set ${\mathcal{B}}$ and the 
continuity of $f$ imply that there exists a $d>0$ such
that
\begin{equation}
\label{2}
B(d,B(\ep,f(x)))\subset f(B(\ep,x))
\end{equation}
for $x\in{\mathcal{B}}$.

Let $\{p_k\in{\mathcal{B}}:\;0\leq k\leq m\}$ be a $d$-pseudotrajectory of $f$;
denote $C_k=B(\ep,p_k)$. 

We claim that
\begin{equation}
\label{3}
C_{k+1}\subset f(C_k),\quad k=0,\dots,m-1.
\end{equation}

Indeed, let $f(C_k)=[s,t]$ and let $y\in C_{k+1}$.
Inclusion (\ref{2}) with $x=p_k$ implies that
$$
f(p_k)+\ep+d\leq t.
$$
Now it follows from the inequalities
$|y-p_{k+1}|\leq\ep$ and $|f(p_k)-p_{k+1}|<d$
that $|y-f(p_k)|<\ep+d$; hence,
$$
y<f(p_k)+\ep+d\leq t.
$$
Similarly one shows that $y>s$, which proves (\ref{3}).

It follows from inclusions (\ref{3}) that the set
$$
C=\cap_{k=0}^m f^{-k}(C_k)
$$
is not empty. Clearly, for any point $r\in C$, inequalities
(\ref{0}) hold.
\end{proof}

In Theorem \ref{res1}, we can say nothing about the dependence of $d$
on $\ep$. For a particular example of a diffeomorphism with a
nonhyperbolic fixed point considered below, 
such a dependence can be clarified.
\medskip

{\bf Example 1. } Let $f(x)=x+x^{2n+1}+X(x)$, where
$n$ is a natural number.

Take $\ep>0$ and let
$$
S(x,\ep):=\frac{(x+\ep)^{2n+1}-x^{2n+1}}{\ep}.
$$
Then $S(x,\ep)$ is a polynomial in $x$ of even degree
with positive leading coefficient.

Since the derivative of $S(x,\ep)$ in $x$ has a unique
zero $x=-\ep/2$, the inequality
$$
S(x,\ep)\geq S\left(\frac{-\ep}{2},\ep\right)=\frac{\ep^{2n}}{2^{2n}}
$$
holds. Thus, the form 
$$
S(x,\ep)-\frac{\ep^{2n}}{1+2^{2n}}
$$
of degree $2n$ is positive definite, and there exists a
positive number $\alpha=\alpha(n)$ such that
\begin{equation}
\label{4}
S(x,\ep)-\frac{\ep^{2n}}{1+2^{2n}}\geq\alpha(x^{2n}+\ep^{2n}).
\end{equation}

Assume that
\begin{equation}
\label{5}
\frac{|X(x+v)-X(x)|}{|v|}=o(x^{2n}+v^{2n}),\quad x,v\to 0.
\end{equation}
Then there exist $A,a>0$ such that if $|x|\leq A$
and $0<\ep<a$, then
\begin{equation}
\label{6}
|X(x+\ep)-X(x)|\leq \frac{\alpha\ep}{2}(x^{2n}+\ep^{2n}).
\end{equation}
Let
\begin{equation}
\label{7}
d=\frac{\ep^{2n+1}}{1+2^{2n}}.
\end{equation}
If $\{p_k\}$ is a $d$-pseudotrajectory with $|p_k|\leq A$,
then it follows from (\ref{4}) and (\ref{6}) that
$$
f(p_k+\ep)-f(p_k)=\ep+\ep S(p_k,\ep)+X(p_k+\ep)-X(p_k)\geq
$$
$$
\geq \ep +\frac{\alpha\ep}{2}(p_k^{2n}+\ep^{2n})+d>\ep+d.
$$
This relation and a similar relation for $f(p_k-\ep)-f(p_k)$
mean that an analog of inclusion (\ref{2}) holds
for any $p_k$.

Repeating the proof of Theorem 1,
we conclude that $f$ has the finite shadowing property 
in the set $B(A,0)$.

Note that, for example, condition (\ref{5}) is valid if $X(x)=x^{2n+2}$.

Our reasoning also shows that if $X(x)\equiv 0$, then $f$ has
the finite shadowing property in the whole line $\R$ with the same
dependence of $d$ on $\ep$ given by (\ref{7}).
\medskip

\begin{remark}In [8], S. Tikhomirov used a different approach to
show that the diffeomorphism $f(x)=x+x^3$ has the shadowing
property with $d=c\ep^3$.
\end{remark}

\section{Lyapunov functions and shadowing}\label{sec3}

As was mentioned in the Introduction, we consider in
detail the problem of finite shadowing for a homeomorphism $f$ of the
plane $\R^2$ in two cases: Case I (the origin is an isolated
nonhyperbolic fixed point) and Case NI (the $y$-axis consists
of fixed points).

We apply the approach based on pairs of Lyapunov type
functions developed in the paper [7].

Let us formulate the sufficient conditions of
finite shadowing obtained in [7] in a form
modified to fit our purposes in this paper.

Let $K_0=\R^2$ in Case I, and let 
$$
K_0=\{(x,y):\;0<|x|<1\}
$$
in Case NI.

We assume that there exist two continuous
nonnegative functions $V$ and $W$ defined on $K_0\times K_0$
such that $V(p,p)=W(p,p)=0$ for any $p\in K_0$
and the conditions (C1)-(C9) stated below are satisfied.

We formulate our conditions not directly in terms of the
functions $W$ and $V$ but in terms of some geometric objects
defined via these functions.

Fix a positive number $\de$ and a point $p\in K_0$ and let
$$
P(\de,p)=\{q\in K_0:\;V(q,p)\leq \de,\;W(q,p)\leq \de\},
$$
$$
Q(\de,p)=\{q\in P(\de,p):\;\;V(q,p)=\de\},\quad 
R(\de,p)=\{q\in P(\de,p):\;\;W(q,p)=\de\},
$$
and
$$
T(\de,p)=\{q\in P(\de,p):\;V(q,p)=0\}.
$$

Set
$$
\mbox{Int}^0\,P(\de,p)=\{q\in P(\de,p):\;V(q,p)<\de,\;W(q,p)<\de\},
$$
$$
\partial^0\,P(\de,p)=Q(\de,p)\cup R(\de,p),
$$
and
$$
\mbox{Int}^0\,Q(\de,p)=\{q\in P(\de,p):\;V(q,p)=\de,\;W(q,p)<\de\}.
$$

Let $K$ be a subset of $K_0$ (in our basic examples, $K$ is a small closed
neigborhood of the origin in Case I and $K$ is a neighborhood of
the origin in $K_0$ in Case NI).

Conditions (C1) -- (C4) contain our assumptions on the
geometry of the sets introduced above. In conditions (C2) -- (C4),
$p\in K_0$, and $\de,\De$ are arbitrary positive numbers such that
$\de<\De$.

(C1) For any $\ep>0$ there exists a $\Delta_0=\De_0(\ep)>0$ such that 
$P(\De_0,p)\subset B(\ep,p)$ for $p\in K$.

(C2) $Q(\de,p)$ is not a retract of  $P(\de,p)$;

(C3) $Q(\de,p)$ is a retract of  
$P(\de,p)\setminus T(\de,p)$;

(C4) there exists a retraction
$$
\sigma:P(\De,p)\to P(\de,p)
$$
such that if $V(q,p)\neq 0$, then $V(\sigma(q),p)\neq 0$.

In the next group of conditions, we state our assumptions on
the behavior of the introduced objects and their images
under the homeomorphism $f$. 

Let $p,q\in K$ and let $0<\de<\De$. 
We say that condition ${\mathcal{W}}(\de,\De,p,q)$ is satisfied if

(C5) $f(P(\de,p))\subset \mbox{Int}^0\, P(\De,q),\quad
f^{-1}(P(\de,q))\subset \mbox{Int}^0\, P(\De,p)$;

(C6) $f(T(\de,p))\subset \mbox{Int}^0\,P(\de,q)$;

(C7) $f(T(\De,p))\cap Q(\de,q)=\emptyset$;

(C8) $f(P(\de,p))\cap \partial^0\, P(\de,q)
\subset \mbox{Int}^0\, Q(\de,q)$;

(C9) $f(Q(\de,p))\cap P(\de,q)=\emptyset$. 
\medskip

\begin{figure}[h]\centering
\def\svgwidth{10cm}
 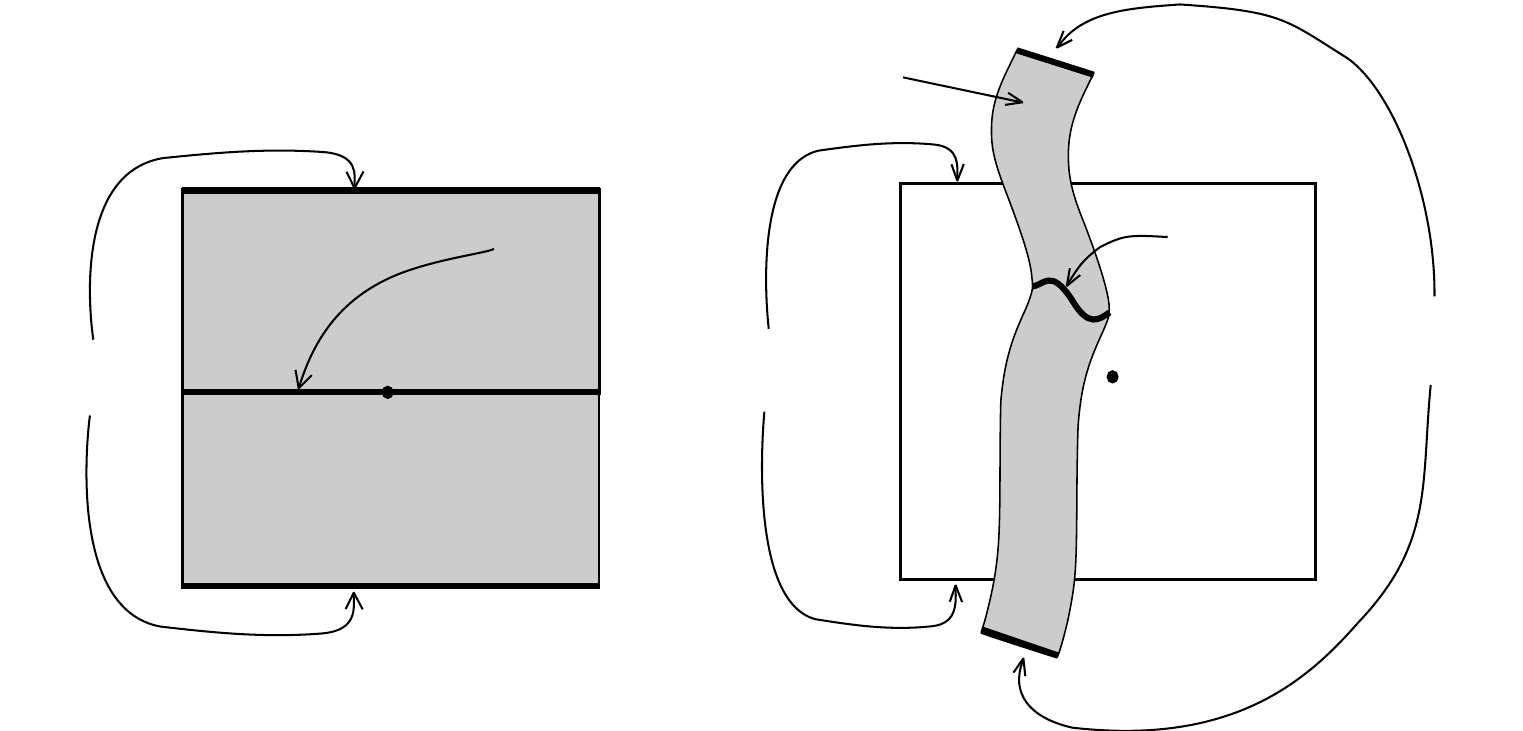
  \caption{Conditions (C6) and (C8)-(C9). Here $P=P(\de,p)$, $\widetilde{P}=P(\de,q)$, $T=T(\de,p)$, $Q=Q(\de,p)$, $\widetilde{Q}=Q(\de,q)$.}
\end{figure}

\begin{figure}[h]\centering
\def\svgwidth{10cm}
 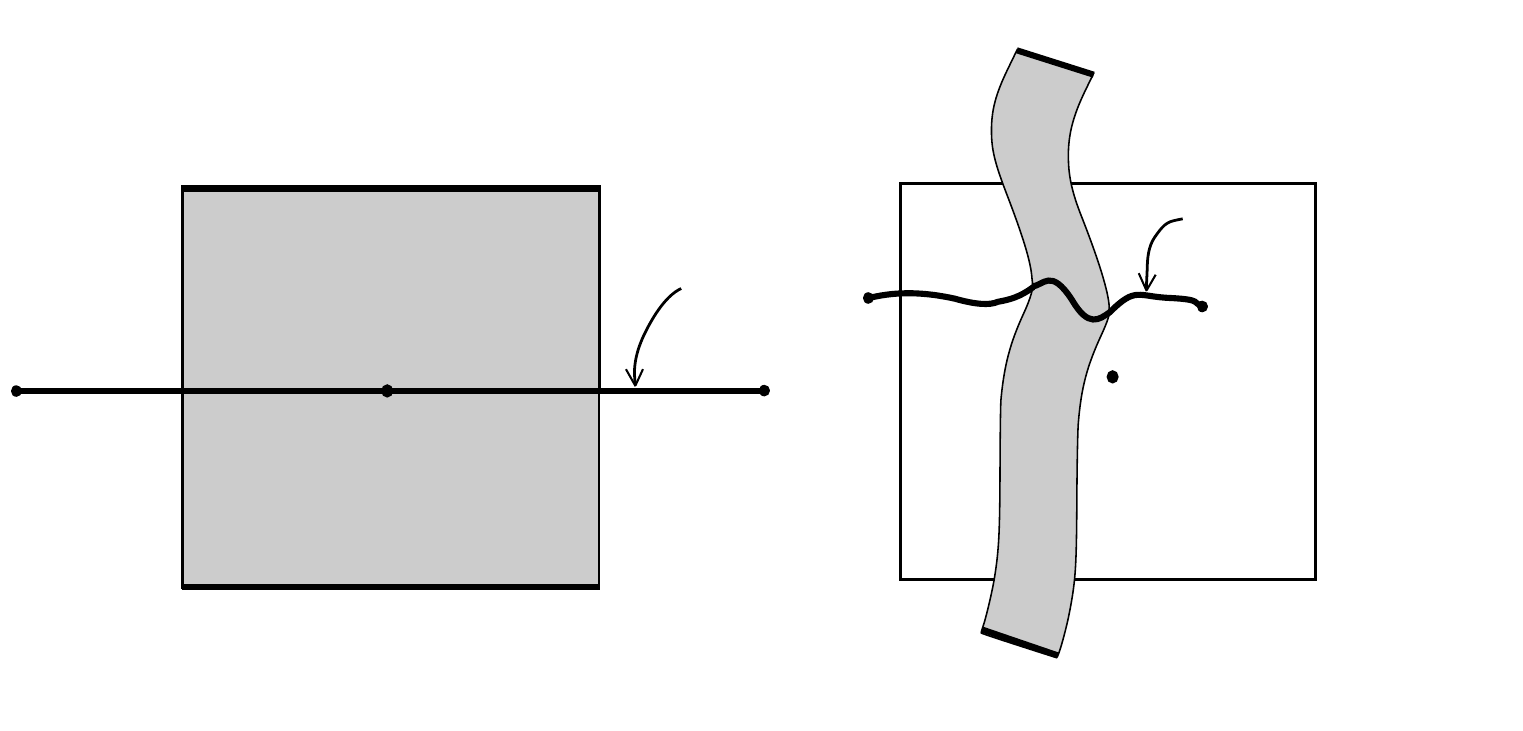
  \caption{Condition (C7). Here $T^\prime=T(p,\De)$.}
\end{figure}

The same reasoning as in [7] proves the following statement.
\begin{proposition}\label{prop}
{ Assume that conditions} {\em (C2) -- (C4) }
{ hold. Let $p_0,\dots,p_m \in K$.
If $0<\de<\De$ and condition ${\mathcal{W}}(\de,\De,p_k,p_{k+1})$ is satisfied
for any $k=0,\dots,m-1$, then there is a point $r\in P(\de,p_0)$
such that}
\begin{equation}
\label{2.1}
f^k(r)\in P(\de,p_k),\quad k=1,\dots,m.
\end{equation}
\end{proposition}

Thus, to show that $f$ has the finite shadowing property in a neighborhood
$K$ of the origin, it is enough to find functions $V$ and $W$ that
satisfy conditions (C1) -- (C4) and to show that for any $\De>0$ there
exists a $\de\in(0,\De)$ with the following property: There exists 
a $d>0$ such that if $p,q\in K$ and $|q-f(p)|\leq d$, then condition
${\mathcal{W}}(\de,\De,p,q)$ is satisfied.

Indeed, take any $\ep>0$, find a corresponding $\De_0$ (see condition
(C1)), then take suitable $\De<\De_0$ and $\de$, and finally find
a $d>0$ having the above property. Then, if
$p_0,\dots,p_m \in K$ is a $d$-pseudotrajectory of $f$, this
pseudotrajectory is $\ep$-shadowed by any point $r$ that
satisfies inclusions (\ref{2.1}).

We realize this scheme in the next section considering Case I.
In Case NI, $f$ does not have the usual shadowing property, and we
have to modify the concept of shadowing (see Sec. \ref{sec5}).

\section{Two-dimensional case. Isolated fixed point}\label{sec4}

Now we consider a two-dimensional homeomorphism
$$
f(x,y)=(g(x,y),h(x,y))
$$
having a fixed point at the origin and assume that $f$
is contracting in the direction of variable $x$
and expanding in the direction of variable $y$
(and both the contraction and expansion are not
assumed to be hyperbolic).

Let $p=(p_x,p_y)$ be the coordinate representation of a point
$p\in\R^2$.

In the case considered, we introduce
two functions,
$$
V(p,q)=|p_y-q_y|\quad\mbox{and}\quad W(p,q)=|p_x-q_x|.
$$

For such functions, conditions (C1) -- (C4) are obviously
satisfied for any $p\in\R^2$ and any $0<\de<\De$.

We first formulate general conditions of finite shadowing
for an arbitrary compact subset $K$ of the plane and then
apply them to our first basic example.
\medskip

{\bf Condition 2. } For any $\De_0>0$ there exist 
$\de,\De>0$ such that $\de<\De<\De_0$ and if
$p\in K$, then condition (C5) with $q=f(p)$ is satisfied,
\begin{equation}
\label{3.1}
|g(p_x+v,p_y+w)-g(p_x,p_y)|<\de\quad\mbox{for}\quad (v,w)\in H(\de),
\end{equation}
where
$$
H(\de)=\{|v|\leq\de,\;w=0\}\cup \{|v|=\de,\;|w|\leq \de\},
$$
\begin{equation}
\label{3.2}
|h(p_x+v,p_y)-h(p_x,p_y)|<\de\quad\mbox{for}\quad 0\leq|v|\leq \De,
\end{equation}
and
\begin{equation}
\label{3.3}
|h(p_x+v,p_y+w)-h(p_x,p_y)|>\de\quad\mbox{for}\quad 0\leq|v|\leq \de,
|w|=\de.
\end{equation}
\begin{theorem}
{\em If $K$ is a compact subset of the plane
and condition} 2 {\em is satisfied, then $f$ has
the finite shadowing property in the set} $K$
\end{theorem}

\begin{proof} First we show that condition 2 implies that
condition ${\mathcal{W}}(\de,\De,p,f(p))$ 
is satisfied for any $p\in K$.

Since
$$
T(\de,p)=\{r:\;|r_x-p_x|\leq \de, r_y=p_y\},
$$
conditions (\ref{3.1}) with $w=0$ and (\ref{3.2}) imply condition (C6).

Similarly, condition (\ref{3.2}) implies condition (C7).

Since condition (\ref{3.3}) holds,
\begin{equation}
\label{3.18}
f(Q(\de,p))\cap P(\de,f(p))=\emptyset.
\end{equation}

Inequalities (\ref{3.1}) for $|v|=\de,\;
0\leq |w|\leq \de$ combined with (\ref{3.18})
show that the image of
the boundary of the set $P(\de,p)$ under the
homeomorphism $f$ does not intersect the set $R(\de,f(p))$.
%

Set   $$S=\{(g(p) +a,h(p) +b)\mid |a|\geq \de, |b|\leq \de  \}.$$
It is obviously that $R(\de,f(p))\subset S$. From $(\ref{3.1})$, $(\ref{3.2})$ and $(\ref{3.3})$ it follows that $f(\partial^0P(\de,p))\cap S=\emptyset$. Because of connectedness of $f(P(\de,p))$ it follows also that $f(P(\de,p))\cap S = \emptyset$ and condition (C8) is proved.

Finally, condition (C9) follows from condition (\ref{3.3}).

Thus, we have shown that conditions (\ref{3.1}) -- (\ref{3.3})
imply that condition ${\mathcal{W}}(\de,\De,p,f(p))$ 
is satisfied for any $p\in K$.

Finally, we note that since $K$ is compact and $f$, $V$, and $W$ are continuous,
the form of conditions (C5) -- (C9) implies that there exists a $d>0$ depending only
on $\de$ and $\De$ such that if $p\in K$ and $|q-f(p)|<d$,
then condition ${\mathcal{W}}(\de,\De,p,q)$ is satisfied.

Now Theorem 4.1 is a corollary of the proposition stated in the 
previous section. 
\end{proof}

{\bf Example 2. } Consider a diffeomorphism (\ref{3.0})
in which $n,m$ are natural numbers
and $X,Y$ are smooth functions that vanish at the origin
together with their Jacobi matrices.

First we fix a small closed neighborhood $K$ of the
origin (in what follows, we make it as small as our future conditions
require).

Let us assume that if $p\in K$, then
\begin{equation}
\label{3.5}
(2n+1)p_x^{2n}-
\frac{\partial X}{\partial x}(p)-\nu\frac{\partial X}{\partial y}(p)>0,
\quad p_x\neq 0,|\nu|\leq 1,
\end{equation}
and
\begin{equation}
\label{3.7}
(2m+1)p_y^{2m}+
\frac{\partial Y}{\partial y}(p)+\nu\frac{\partial Y}{\partial x}(p)>0,
\quad p_y\neq 0,|\nu|\leq 1.
\end{equation}

We take the same functions $V$ and $W$ as above in this section.

It follows from the form of $f$ that for any $\alpha>0$ we can find
a neighborhood $K$ such that
\begin{equation}
\label{3.4}
\left\|\frac{\partial f}{\partial(x,y)}(p)\right\|,
\left\|\frac{\partial f^{-1}}{\partial(x,y)}(p)\right\|
\leq 1+\alpha,\quad p\in K.
\end{equation}

Thus, if $\alpha>0$ is given, then, for
$\de$ small enough (and $K$ properly chosen),
condition (C5) is satisfied with $\De=(1+\alpha)\de$.

We take $\alpha<1$ (in which case we may take $\De=2\de$
in condition (C5))
and assume that $K$ is so small that
$$
\left|\frac{\partial Y}{\partial x}(p)\right|\leq \frac{1}{4},
\quad p\in K.
$$
If $p\in K$ and $|v|\leq\De=2\de$, then
$$
|h(p_x+v,p_y)-h(p_x,p_y)|=|Y(p_x+v,p_y)-Y(p_x,p_y)|\leq
\frac{1}{4}|v|\leq \frac{\de}{2};
$$
thus, condition (\ref{3.2}) is satisfied.

Now let us check condition (\ref{3.3}).
If $0\leq|v|\leq \de$ and $|w|=\de$, then
$v=\nu w$ for some $|\nu|\leq 1$.

Assume, for definiteness, that $w>0$ and estimate,
using condition (\ref{3.7}):
$$
h(p_x+\nu w,p_y+w)-h(p_x,p_y)=\int_0^w\frac{d}{dt}
h(p_x+\nu t,p_y+t)dt=
$$
$$
=\int_0^w[\nu\frac{\partial h}{\partial x}(p_x+\nu t,p_y+t)+
\frac{\partial h}{\partial y}(p_x+\nu t,p_y+t)]dt=
$$
$$
=\int_0^w[\nu\frac{\partial Y}{\partial x}(p_x+\nu t,p_y+t)+
\frac{\partial Y}{\partial y}(p_x+\nu t,p_y+t)+
$$
$$
+1+(2m+1)(p_y+t)^{2m}]dt>w\geq\de,
$$
where we take into account that $p_y+t$ is not identically zero.
This proves condition (\ref{3.3}) (the case $w<0$ is treated
similarly).

To prove condition (\ref{3.1}), we consider the case
where $|v|=\de,\;|w|\leq d$; the case where $w=0$
is treated similarly (note that in both cases, $w=\nu v$
with $|\nu|\leq 1$).

We assume, for definiteness, that $v=\de$, represent
$w=\nu v$ with $|\nu|\leq 1$, and estimate, using 
condition (\ref{3.5}):
$$
g(p_x+\de,p_y+\nu \de)-g(p_x,p_y)=\int_0^\de\frac{d}{dt}
g(p_x+ t,p_y+\nu t)dt=
$$
$$
=\int_0^\de[1-(2n+1)(p_x+t)^{2n}+
\frac{\partial X}{\partial x}(p_x+ t,p_y+\nu t)+
$$
$$
+\nu\frac{\partial X}{\partial y}(p_x+ t,p_y+\nu t)
]dt<\de.
$$
The case $v=-\de$ is treated similarly.
\medskip

Let us consider as a ``test perturbation" a monomial $Y(x,y)=ax^ky^l$
in (\ref{3.0}), where $a\in\R$, $k\geq 0$, and $l\geq 1$. 
Let $(x,y)$ be a point in a
neighborhood $K$ of the origin with $y\neq 0$.

Taking $\nu=0$ in (\ref{3.7}) and dividing the result by
$y^{2m}$, we get the following necessary condition:
$$
alx^ky^{l-2m-1}>-(2m+1).
$$
This condition is obviously satisfied in a small $K$
if $l\geq 2m+1$ (and if $|a|$ is small in the case 
where $l=2m+1$ and $k=0$).

If $l<2m+1$, then the necessary condition looks as
follows: $a>0$, $k$ is even, and $l$ is odd.

Let us write condition (\ref{3.7}) in $K$
in the form
\begin{equation}
\label{3.10}
(2m+1)y^{2m}+alx^ky^{l-1}>|akx^{k-1}y^l|.
\end{equation}

If $l\geq 2m+1$ (and if $|a|$ is small in the case 
where $l=2m+1$ and $k=0$), condition (\ref{3.7})
is obviously satisfied in a small neighborhood of the origin.

In the case where $l\leq 2m$, we get one more necessary
condition: $k+l\geq 2m+1$. Indeed, since $a>0$, $k$ is even, and $l$ is odd,
it is enough to consider (\ref{3.7}) for $x,y\geq 0$.

Now let us write (\ref{3.7}) in the form
$$
ax^{k-1}y^{l-1}(ky-lx)<(2m+1)y^{2m}.
$$
If $k+l<2m+1$, set $x=\frac{kb}{2l}$ and $y=b$ with small $b>0$.
We get an inequality of the form
$$
0<\mbox{const}<b^{2m-k-l+1}
$$
which cannot be satisfied for all small $b$.

Elementary calculations show that if $l\leq 2m$, 
$a>0$, $k$ is even, $l$ is odd, and $k+l\geq 2m+1$,
then condition (\ref{3.10}) is satisfied in a
small neighborhood of the origin.

We can obtain similar conditions if $X(x,y)$ in (\ref{3.0})
is also a monomial.

Our methods allow us to estimate the dependence of $d$ on $\ep$
in the finite shadowing property for the considered case.

For example, if $X(x,y)=a_1x^{k_1}y^{l_1}$ and
$Y(x,y)=a_2x^{k_2}y^{l_2}$ in (\ref{3.0}) with $k_1>2n+1$
and $l_2>2m+1$, then
the same reasoning as above shows that there exists a
neighborhood $K$ of the origin and  a small number $c>0$
such that if a $\de>0$ is given and $\{p_k\}$ is a finite set of points in $K$
with 
$$
|f(p_k)-p_{k+1}|\leq c\de^p,
$$
where $p=\max(2n+1,2m+1)$, then conditions ${\mathcal{W}}(\de,2\de,p_k,p_{k+1})$
are satisfied. This means that $f$ has in $K$ the finite shadowing
property with the following dependence of $d$ on $\ep$: $d=c\ep^p$.

\section{Two-dimensional case. Nonisolated fixed point}\label{sec5}

In this section, we consider a model example of the diffeomorphism (\ref{4.0})
for which the origin is a nonisolated fixed point (any point
$(0,y)$ is a fixed one).

Of course, $f$ does not have the shadowing property.

Nevertheless, we show that $f$ has an analog of the finite
shadowing property if we consider pseudotrajectories $\{p_k\}$
with $(p_k)_x\neq 0$ and allow the ``errors" 
$$
|f(p_k)-p_{k+1}|
$$
to depend on $(p_k)_x$. The errors must be smaller for smaller
values of $|(p_k)_x|$. Such an approach (in the case of a nontransverse
homoclinic point) has been suggested by S. Tikhomirov.

We restrict our consideration to the case of a diffeomorphism
$f$ of a very simple form (\ref{4.0}) to simplify presentation
(as the reader will see, even this case is not completely trivial);
of course, our reasoning can be applied in more general
situations.

Note that 
\begin{equation}
\label{4.01}
f^{-1}(x,y)=\left(2x,\frac{y}{1+4x^2}\right).
\end{equation}

Thus, we consider a finite pseudotrajectory $p_0,\dots,p_m$
and assume that $(p_k)_x\neq 0$ and
\begin{equation}
\label{4.1}
|f(p_k)-p_{k+1}|\leq d(p_k)_x^2,\quad k=0,\dots,m-1,
\end{equation}
for some $d>0$.

Our main result is as follows. Recall that in our case,
$$
K_0=\{(x,y):\;0<|x|<1\}.
$$

\begin{theorem}\label{theor5}
{ There exists a neighborhood $K$ of the
origin and a number $c>0$ such that, for any $\ep>0$ and
any pseudotrajectory $p_0,\dots,p_m$ in $K\cap K_0$
that satisfies conditions} {\em (\ref{4.1})} { with
$d=c\ep$ there exists a point $r$ for which inequalities}
{\em (\ref{0})} { are satisfied.}
\end{theorem}
\begin{proof}
 As above, in our proof we use the approach
based on Lyapunov functions, but now one of the functions
is modified. We take
$$
V(q,p)=|p_y-q_y|\quad\mbox{and}\quad
W(q,p)=\frac{|p_x-q_x|}{|p_x|(1-|p_x|)}.
$$
Clearly, these functions are nonnegative and continuous
on $K_0\times K_0$ and they vanish if their arguments
coincide.

It is obvious that conditions (C1) -- (C4) are satisfied
(and we can take $\De_0(\ep)=\ep/2$ in condition (C1)).

In the following proof, we take
\begin{equation}
\label{4.3}
\de=\frac{\De}{N}=\frac{\De_0}{N}=\frac{\ep}{2N}
\end{equation}
(the constant $N$ is chosen below)
and $d=c\de$ in condition (\ref{4.1}). 

First we take $c=1$ and then make $c$ smaller
preserving the same notation $c$; the same is done
with the neighborhood $K$.

Our main goal is to check condition ${\mathcal{W}}(\de,\De,p_k,p_{k+1})$
for consecutive points $p_k,p_{k+1}$ of the pseudotrajectory
considered for properly chosen $\de$ and $\De$; after that, we
apply Proposition \ref{prop} stated in Sec. \ref{sec3}.

To make the presentation shorter, let $p_k=p=(x,y)$
and $p_{k+1}=q=(x',y')$.

Thus, we may assume that $|x|,|x'|,|y|,|y'|$ are as
small as we need.

First we claim that there exists a number $N>0$
such that if $K$ is a small neigborhood of the origin,
$p,q\in K$,
and $\de<1$, then
inclusions (C5) hold with $\De=N\de$.

A point $(x+v,y+w)$ belongs to $P(\de,p)$ if and only if
$|v|\leq\de|x|(1-|x|)$ and $|w|\leq\de$. Similar inequalities
define the set $P(\de,q)$.

Thus, to prove our claim, we have to show that there
exists a number $N$ such that if $|v|\leq\de|x|(1-|x|)$ and $|w|\leq\de$,
then
$$
\left|\frac{x+v}{2}-x'\right|< N\de|x'|(1-|x'|)
$$
and
$$
|(y+w)(1+(x+v)^2)-y'|< N\de.
$$
In addition, we have to show that if
$|v|\leq\de|x'|(1-|x'|)$ and $|w|\leq\de$, then
$$
\left|2(x'+v)-x\right|< N\de|x|(1-|x|)
$$
and
$$
\left|\frac{y'+w}{1+4(x'+v)^2}-y\right|< N\de
$$
(see the formula (\ref{4.01}) for $f^{-1}$).

Let us prove our statement on the existence of a number $N$ 
for which the third of the above inequalities holds
(the remaining inequalities are established using
a similar reasoning).

We may assume that $|x|<1/4$ and $\de<1$. Then
it follows from (\ref{4.1}) that
$$
|x'-x/2|<dx^2<\de|x|/4<|x|/4,
$$
which implies that $|x'|<|x|$ and $|2x'-x|<\de|x|/2$.
In addition,
$$
|v|\leq\de|x'|(1-|x'|)<\de|x'|<\de|x|
$$
and 
$$
\de|x|(1-|x|)>\de|x|/2.
$$

Combining these inequalities, we see that if
$N=3$, then
$$
N\de|x|(1-|x|)>3\de|x|/2>|2x'-x|+|v|,
$$ 
as required.

In what follows, we take $\De=N\de$ with a fixed $N$.

Now we check conditions (C6) and (C7). To simplify
presentation, we assume that $x>0$ (and $x$ is as small as
we need).

We note that
$$
f(T(\de,p))=\left\{\left(\frac{x+v}{2},y(1+(x+v)^2)\right):\;|v|\leq\de x(1-x)\right\}.
$$
Thus, the projection of $f(T(\de,p))$ to the $x$ axis is the segment $[A^-,A^+]$,
where
$$
A^-=\frac{x}{2}-\frac{\de x}{2}(1-x)\quad\mbox{and}\quad
A^+=\frac{x}{2}+\frac{\de x}{2}(1-x).
$$
At the same time, if $x'=x/2+u$, then
the projection of $P(\de,q)$ to the $x$ axis
is the segment $[B^+(u),B^-(u)]$, where
$$
B^-(u)=\frac{x}{2}+u-\frac{\de(x+2u)}{2}\left(1-\frac{x}{2}-u\right),
$$
$$
B^+(u)=\frac{x}{2}+u+\frac{\de(x+2u)}{2}\left(1-\frac{x}{2}-u\right),
$$
and $|u|\leq dx^2$ (in the above formulas, we note that if $d$ and
$x$ are small, then $x+2u>0$).

Since 
$$
B^+(0)=A^++\frac{\de x^2}{4}\quad\mbox{and}\quad
B^-(0)=A^--\frac{\de x^2}{4},
$$
it is easy to understand that there exists a $c>0$
(independent of $x$ and $\de$) such that if
$d\leq c\de$ in (\ref{4.1}), then 
\begin{equation}
\label{4.2}
[A^+,A^-]\subset (B^-,B^+).
\end{equation}

To complete the proof of condition (C6) and prove condition (C7), we 
note that 
$$
f(T(\De,p))=\left\{\left(\frac{x+v}{2},y(1+(x+v)^2)\right):\;|v|\leq N\de x(1-x)\right\}.
$$
If $y'=y(1+x^2)+u$, then the $y$ coordinate of any point of the set $Q(\de,q)$ 
is either $y'-\de$ or $y'+\de$.

Let us represent
$$
y'+\de-y(1+(x+v)^2)=\de-y(2xv+v^2)+u.
$$
Since $|v|\leq N\de x(1-x)$, we have the estimates
$$
|xv|\leq N\de x^2(1-x),\quad v^2\leq N^2\de^2x^2(1-x)^2.
$$
If the neighborhood $K$ is small (so that $|x|$ and $|y|$
are small enough), we conclude from the inequality
$|u|\leq c\de x^2$ (with $c$ fixed above) that
$$
y'+\de-y(1+(x+v)^2)>0,\quad |v|\leq N\de x(1-x),
$$
for small $\de$,
i.e., $f(T(\De,p))$ does not intersect the ``upper" component of
$Q(\de,q)$. A similar reasoning is applicable to the ``lower" component.
This proves condition (C7) (and, combined with inclusion (\ref{4.2}),
condition (C6)).

It remains to check conditions (C8) and (C9). 
Let us start with condition (C9). Consider a point $r=(x+v,y+\de)$
of the ``upper" component of $Q(\de,p)$. In this case, $|v|\leq \de x(1-x)$.
The $y$ component of the point $f(r)$ is
$$
(y+\de)(1+(x+v)^2).
$$
If $y'=y(1+x^2)+u$, then the projection of the set $P(\de,q)$ 
to the $y$ axis is the segment $D=[y(1+x^2)+u-\de,y(1+x^2)+u+\de]$.

Let us represent
$$
(y+\de)(1+(x+v)^2)-y(1+x^2)-u-\de=
$$
$$
=\de(x+v)^2+y(2xv+v^2)-u.
$$
The estimates
$$
\de(x+v)^2\geq \de x^2(1-\de)^2
$$
and
$$
|y(2xv+v^2)|\leq |y|(2\de x^2+\de^2x^2)\leq \de x^2/2
$$
(which is valid if $|y|$ and $\de$ do not exceed a small
value independent of $x$) imply that there exists a
constant $c$ (inependent of $x$, $y$, and $\de$) such that
if $u\leq c\de x^2$, then the projection of the point $f(r)$
to the $y$ axis does not belong to the segment $D$.

Applying a similar reasoning to points $r=(x+v,y-\de)$,
we complete the proof of condition (C9).
Estimates of a similar form prove condition (C8).

To complete the proof of Theorem \ref{theor5}, we take into account 
the equality $d=c\de$ and relations (\ref{4.3}).
\end{proof}


\medskip
Received xxxx 20xx; revised xxxx 20xx.
\medskip


\begin{thebibliography}{99}

\bibitem{A11}
     \newblock S. Yu. Pilyugin, 
     \newblock \emph{Shadowing in dynamical systems},
     \newblock Lecture Notes in Mathematics, Springer \textbf{1706} (1999).

\bibitem{Zeng}
    \newblock K. Palmer,
    \newblock  {\em Shadowing in dynamical systems. Theory and applications},
    \newblock Kluwer (2000).

\bibitem{Whho}
\newblock S. Yu. Pilyugin,
\newblock {\em Theory of pseudo-orbit shadowing
in dynamical systems,}
\newblock Differential Equations {\bf 47} (2011), 1929-1938.

\bibitem{tachet}
\newblock S. M. Hammel, J. A. Yorke,  and C. Grebogi, 
\newblock {\em Numerical orbits of chaotic dynamical processes represent
true orbits,}
\newblock Bull. Amer. Math. Soc. {\bf 19} (1988), 465-470.

\bibitem{lalala}
\newblock Judy Kennedy, James A. Yorke,
\newblock {\em Shadowing in Higher Dimensions,}
\newblock Progress in Nonlinear Differential Equations and Their Applications Volume {\bf 75} (2008),  241-246.

\bibitem{lala}
\newblock J. Lewowicz,
\newblock {\em Lyapunov functions and topological stability,}
\newblock J. Differential Equations {\bf 38} (1980), 192-209.


\bibitem{la}
\newblock A. A. Petrov, S. Yu. Pilyugin,
\newblock {\em Lyapunov functions, shadowing, and topological stability,}
\newblock Topol. Methods Nonlin. Anal. (submitted).

\bibitem{l}
\newblock Sergey Tikhomirov,
\newblock {\em Holder Shadowing on Finite Intervals,}
\newblock \arXiv{1106.4053v2}.


\end{thebibliography}
\end{document}